\documentclass[12pt,a4paper]{article}
\usepackage[utf8]{inputenc}
\usepackage[T1]{fontenc}

\usepackage{enumerate}

\usepackage{xcolor}
\usepackage{graphicx}
\usepackage{enumerate}

\usepackage{amsthm,hyperref,lmodern,mathrsfs,amssymb,amsmath,amsfonts}

\usepackage{comment}

\usepackage{geometry}
\newtheorem{thm}{Theorem}
\newtheorem{rmq}{Remark}

\newtheorem{defi}{Definition}
\newtheorem{prop}{Proposition}
\newtheorem{corol}{Corollary}

\newcommand{\C}{\mathcal{C}}

\newcommand{\St}{\mathcal{S}}
\newcommand{\R}{\mathbb{R}}
\newcommand{\E}{\mathbb{E}}

\newcommand{\Jac}{\mathrm{Jac}}

\title{Isoperimetric and geometric inequalities in quantitative form: Stein's method approach}
\author{Jordan Serres \thanks{Institut de Mathématiques de Toulouse, France, jserres@insa-toulouse.fr} }

\begin{document}
\maketitle

\begin{center}
\textbf{Abstract}
\end{center}
\begin{quotation}
We adapt Stein's method to isoperimetric and geometric inequalities. The main challenge is the treatment of boundary terms. We address this by using an elliptic PDE with an oblique boundary condition. We apply our geometric formulation of Stein's method to obtain stability of the Brock-Weinstock inequality, stability of the isoperimetric inequality under a constraint on Steklov's first non-zero eigenvalue, and stability for the combination of weighted and unweighted perimeters. All stability results are formulated with respect to the $\alpha$-Zolotarev distance, $\alpha\in(0,1]$, that we introduce to interpolate between the Fraenkel asymmetry and the Kantorovich distance.
\end{quotation}

\section{Introduction}

\subsection{Isoperimetric and geometric inequalities in quantitative form}

In recent years, a great deal of research has gone into obtaining quantitative versions of well-known geometric inequalities. These type of results, known as stability results for functional and geometric inequalities, are examples of inverse problems. The general form of these problems is as follows. We start with an inequality whose extremizers are known, and ask whether the deficit term in this inequality can be bounded below by a certain distance to the set of extremizers, see \cite{figalli2013stability} for a more detailed presentation of the topic. Stability results have been obtained for a large number of functional and geometric inequalities, such as the Faber-Krahn and the Szegö-Weinberger inequalities (see the survey \cite{brasco2016spectral}), the Poincaré inequality (see e.g. \cite{CFP, FGS, jordan2023stability}), or the Sobolev inequality (see e.g. \cite{cianchi2009sharp, nobili-violo}) to name but a few.
In this paper, we focus on the quantitative form of the isoperimetric inequality and the Brock-Weinstock inequality. The isoperimetric inequality is the famous result that among all possible shapes of a given volume, the one with the smallest area is a ball. Expressed by a formula, this gives that for any sufficiently smooth $\Omega\subset\R^d$, it holds that
\[
\frac{|\partial\Omega|}{|\Omega|^{\frac{d-1}{d}}} \geq \frac{|\partial B|}{|B|^{\frac{d-1}{d}}},
\]
where $B$ stands for a ball with same volume as $\Omega$, and $|\partial \cdot |$ denotes the $(d-1)$-Hausdorff measure of the boundary of a set. The problem of the stability of the isoperimetric inequality consists in showing that, when $|\Omega|=|B_1|$ has the same volume as the standard unit ball $B_1$, the isoperimetric deficit $|\partial\Omega|-|\partial B_1|$ controls a certain distance to $B_1$. The history of the stability of the isoperimetric inequality began with Fuglede's perturbative result \cite{fuglede1989stability}, followed by Hall's result for sets of finite perimeter \cite{hall1992quantitative}, which was refined with a sharp exponent by Fusco, Maggi and Pratelli \cite{FMP08}, based on a symmetrization method. The same results were then derived using optimal transport techniques by Figalli, Maggi and Pratelli \cite{figalli2010mass}, which were also used to prove quantitative anisotropic isoperimetric inequalities. We refer the reader to Fusco's survey \cite{fusco2015quantitativesurvey} for a detailed overview on the topic. The quantity used to measure the distance between $\Omega$ and the ball is the following $L^1$-distance between sets, known as the Fraenkel asymmetry
\begin{equation}\label{eq:fraenkel}
\alpha(\Omega) := \min_{x\in\R^d} \left\{ \frac{|\Omega\Delta B_r(x)|}{|B_r(x)|}\,;\, |B_r(x)|=|\Omega| \right\},
\end{equation}
where $B_r(x)$ stands for the ball centered at $x$ of radius $r>0$, and $\Delta$ denotes the symmetric difference of two sets.\newline \newline
{\textbf{Theorem.} (\cite{FMP08} Fusco-Maggi-Pratelli, 2008)\\
There exist a constant $C_d$ depending only on the dimension $d$, such that for any measurable set $\Omega$ of finite measure and normalized such that $|\Omega|=|B_1|$,
\[
|\partial\Omega| - |\partial B_1| \geq C_d\, \alpha(\Omega)^2.
\]
}
This result was also improved by Fusco and Julin \cite{fuscojulin14}, where they replaced the Fraenkel asymmetry by the oscillation index, defined as
\begin{equation}\label{eq:oscillationindex}
\beta(\Omega) := \min_{y\in\R^d} \left\{ \left(\int_{\partial E} \left| \nu - \nu_{r,y}(\pi_{y,r}(x)) \right|^2 d\mathcal{H}^{d-1}(x) \right)^{1/2} \right\}
\end{equation}
where $\nu$ denotes the unit outward-pointing normal vector of $\partial\Omega$, $\nu_{r,y}$ denotes the unit outward-pointing normal vector of $\partial B_r(y)$, $\pi_{y,r}(x)$ denotes the projection of $x$ on the boundary of $B_r(y)$, and $\mathcal{H}^{d-1}$ stands for the $(d-1)$-dimensional Hausdorff measure. 
The isoperimetric inequality also admits weighted versions, were the usual perimeter is replaced by some weighted perimeter
\[
\mathrm{Per}_{f} (\partial \Omega) := \int_{\partial\Omega} f(x)\,d\mathcal{H}^{d-1}(x)
\]
for some weight function $f:\R^d\to\R_+$. Weighted isoperimetric problems, also known as isoperimetric problems with density \cite{morgan2005manifolds}, are an active research topic that has recently been the subject of numerous studies.
To mention just a few, let us mention the works of Rosales, Canete, Bayle and Morgan \cite{rosales2008isoperimetric} and also Morgan and Pratelli \cite{morgan2013existence} on the existence of isoperimetric regions in $\R^n$ with density, the results of Canete, Miranda and Vittone for isoperimetric problems in planes with density \cite{canete2010some}, or Csató's work in the case of weight functions of the form $|x|^p$ in the plane \cite{csato}.
We should also mention Borell's seminal paper on the case where the density is Gaussian \cite{borell1975brunn},
the results of Bobkov and Ledoux in the case of densities that are $\kappa$-concave probability measures \cite{bobkov2009weighted},
Milman's work in the case of a density satisfying a curvature dimension condition \cite{milman2015sharp},
Cabré, Ros-Oton and Serra's use of the ABP method to derive sharp isoperimetric inequalities in the context of a homogeneous weight in a convex cone \cite{cabre2016sharp}
and also Chambers' proof of the log-convex density conjecture \cite{chambers2019proof}.
The perimeter with density $f(x)=|x|^2$ is known as the boundary momentum and  its importance can be emphasized because of its connection with Steklov's spectral problem. The sharp isoperimetric inequality for the boundary momentum was proved by Betta, Brock, Mercaldo and Posteraro in \cite{betta1999weighted}, and states that for all $\Omega$ normalized such that $|\Omega|=|B_1|$,
\begin{equation}\label{eq:weightedisop}
\mathrm{Per}_{|x|^2} (\partial \Omega) \geq \mathrm{Per}_{|x|^2} (\partial B_1) 
\end{equation}
Note that in the case of the standard unit ball $B_1$, the weighted and unweighted perimeters coincide: $\mathrm{Per}_{|x|^2} (\partial B_1) = |\partial B_1|$. 
Quantative forms for weighted isoperimetric inequalities is a recent and active research topic.
These include the results of Brasco, De Philippis and Ruffini under certain convexity assumptions on the weight function \cite{BrascoDePhilippisRuffini},
the work of Cinti, Glaudo, Pratelli, Ros-Oton and Serra in the case of convex cones with homogeneous weights \cite{cinti2022sharp},
and Fusco and La Manna's stability result for log-convex densities \cite{fusco2023some}.
In particular, Brasco, De Philippis and Ruffini's theorem applies to the boundary momentum and gives the following stability estimate for Inequality \eqref{eq:weightedisop}.\newline \newline
{\textbf{Theorem.} (\cite{BrascoDePhilippisRuffini} Brasco-De Philippis-Ruffini, 2012)\\
There exist a constant $C_d$ depending only on the dimension $d$, such that for any open bounded Lipschitz set $\Omega$ normalized such that $|\Omega|=|B_1|$,
\[
\int_{\partial\Omega} |x|^2\,d\mathcal{H}^{d-1}(x) \geq \int_{\partial B_1} |x|^2\,d\mathcal{H}^{d-1}(x)\left[1 + C_d \left(\frac{|\Omega\Delta B_1|}{|B_1|}\right)^2 \right],
\]
where $\Delta$ denotes the symmetric difference of sets.\newline\newline
}
Among spectral geometric inequalities, the most famous is probably the Faber-Krahn inequality \cite{faber1923beweis, krahn1925rayleigh}, which states that the first Dirichlet eigenvalue of a domain $\Omega\subset\R^d$ is always larger than the corresponding eigenvalue for a ball of the same volume.
The study of the stability of such inequalities goes back to the work of Hansen and Nadirashvili \cite{hansen1994isoperimetric} and Melas \cite{melas1992stability}. Hansen and Nadirashvili showed that when the first Dirichlet eigenvalue of a domain is close to that of a ball of the same volume, then this domain is close to being a ball, in the sense that the ratio of its inradius to its circumradius is close to one, and Melas showed the same type of estimate in $L^1$ distance in the case of a bounded convex domain. Bhattacharya and Weitsman \cite{weitsmanbhattacharya} and independently Nadirashvili \cite{nadirashvili1996conformal} conjectured that the deficit in the Faber-Krahn inequality controls the square of the Fraenkel asymmetry \eqref{eq:fraenkel}. This was partially solved by Bhattacharya \cite{bhattacharya2001some}, Fusco, Maggi and Pratelli \cite{fusco2009stabilityFK} and also Povel \cite{povel1999confinement} and Sznitman \cite{sznitman1997fluctuations} with probabilistic applications. The conjecture has been fully solved by Brasco, De Philippis and Velichkov \cite{10.1215/00127094-3120167}, proving the following. \newline\newline
{\textbf{Theorem.} (\cite{10.1215/00127094-3120167} Brasco-De Philippis-Velichkov, 2015)\\
There exist a constant $C_d$ depending only on the dimension $d$, such that for any open set $\Omega$ with finite measure and normalized such that $|\Omega|=|B_1|$,
\[
\lambda_1(\Omega) \geq \lambda_1(B_1) + C_d\, \alpha(\Omega)^2,
\]
where $\lambda_1$ denotes the first Dirichlet eigenvalue and $\alpha$ denotes the Fraenkel asymmetry \eqref{eq:fraenkel}.\newline\newline
}
The Steklov spectral problem on a domain $\Omega\subset\R^d$ is given by
\[
\left\{
    \begin{array}{ll}
        \Delta u_k(x) = 0, & x\in \Omega \\
        \nabla u_k(x) \cdot \nu(x) = \sigma_k\,u_k(x),\, & x\in \partial \Omega
    \end{array}
\right.
\]
where $\Delta$ is the Laplacian, the $u_k$ are the Steklov eigenfunctions of $\Omega$, and the eigenvalues $(\sigma_k)_{k\geq 0}$ are its Steklov spetrum. The functions $u_k$ correspond to the harmonic extensions of the eigenfunctions of the Dirichlet-to-Neumann operator $\mathcal{D}: \C^\infty(\partial\Omega)\to \C^\infty(\partial\Omega)$, defined by $\mathcal{D}(f) = \nabla f\cdot \nu$. For physical motivation and context, we refer the reader to \cite{kuznetsov2014legacy}, and for recent developments and open questions, we refer to the survey \cite{colbois2024some}. The first eigenvalue $\sigma_0$ is zero and its eigenspace is formed by the constant functions. Of greater importance is the first non-zero Stecklov eigenvalue $\sigma_1>0$, which is solution of the following variational principle
\[
\sigma_1 = \min\left\{ \int_\Omega |\nabla u(x)|^2\, dx\,:\, \int_{\partial\Omega} u(x)^2 d\mathcal{H}^{d-1}(x)=1, \int_{\partial\Omega} u(x) d\mathcal{H}^{d-1}(x)=0, u\in H^1(\Omega) \right\}.
\]
In particular, we recognize the Sobolev trace inequality
\[
\sigma_1\,\int_{\partial\Omega} u(x)^2 d\mathcal{H}^{d-1}(x) \leq \int_\Omega |\nabla u(x)|^2\, dx,
\]
which holds for all $u\in H^1(\Omega)$ centered as $\int_{\partial\Omega} u(x) d\mathcal{H}^{d-1}(x)=0$. Among all sufficiently smooth domains $\Omega$ having the same volume as the standard unit ball $B_1$, the one with the largest first non-zero Steklov eigenvalue $\sigma_1$ is the ball $B_1$. This result is known as the Brock-Weinstock inequality. Proved in dimension $2$ by Weinstock \cite{weinstock1954inequalities} and generalized to any dimension by Brock \cite{brock2001isoperimetric}, it can be written as
\[
\sigma_1(\Omega)\leq 1=\sigma_1(B_1),
\]
for all domains $\Omega$ with Lipschitz boundary and such that $|\Omega|=|B_1|$.
Furthermore, equality is attained if, and only if $\Omega$ is a ball of radius $1$. The stability of the Brock-Weinstock inequality has been adressed by Brasco, De Philipis and Ruffini in \cite{BrascoDePhilippisRuffini}, where they proved the following.\newline\newline
{\textbf{Theorem.} (\cite{BrascoDePhilippisRuffini} Brasco-De Philippis-Ruffini, 2012)\\
There is a constant $C_d$ depending only on the dimension $d$, such that for any open bounded Lipschitz set $\Omega$ normalized such that $|\Omega|=|B_1|$,
\[
1- \sigma_1(\Omega)\geq C_d\,\alpha(\Omega)^2,
\]
where $\alpha$ denotes the Fraenkel asymmetry.
}

\subsection{Stein's method}

Stein's method is a set of techniques introduced by Stein \cite{stein1972bound} to give a quantitative convergence rate in the central limit theorem, without any use of the Fourier transform. Here we briefly present the use of the method in stability problems, and refer the reader to Section \ref{sec:fundamStein} for a more technical introduction. The main idea of Stein's method is that we can control the distance between a probability distribution and the standard normal distribution by means of a deficit in an integration-by-parts formula. Noting that in the right cases, the Euler-Lagrange equation of a variational problem characterizes the extremizers, Courtade and Fathi \cite{courtadefathi20} pointed out that the almost extremizers should almost satisfy the Euler-Lagrange equation, and so the stability problem of showing that the deficit in the Euler-Lagrange equation controls some appropriate notion of distance to the set of extremizers coincides exactly with the framework of Stein's method. This line of thought was original used to derive stability results for the Poincaré constant and the logarithmic Sobolev constant of uniformly log-concave probability distributions \cite{courtadefathi20}, and then was followed by stability results for the Poincaré constant of a probability distribution under moment constraints in $\R^d$ \cite{CFP}, in the general setting of a Markov diffusion triple \cite{jordan2023stability}, and under a curvature-dimension condition \cite{FGS}. We can also mention the stability results for the Poincaré-Korn inequality \cite{CFPoincareKorn}, for Klartag's improved Lichnerowicz inequality \cite{CFimprovedLichn}, or for the eigenvalues of any order of a one-dimensional diffusion \cite{serresSPA}. Due to technical problems with boundary terms, this method had never been implemented before this work for functional and geometric inequalities involving shapes instead of probability distributions.

\subsection{Main results}\label{sec:mainresults}

The goal of this paper is to develop Stein's method for shapes. More precisely, we adapt the usual Stein's method in the case where probability distributions supported on the whole space $\R^d$ are replaced by bounded domains $\Omega\subset\R^d$. The main technical challenge is the presence of a boundary, and the fact that there is no standard method for extending solutions with Neumann conditions from one general domain $\Omega$ to another. We address this difficulty by replacing the Neumann condition with a more general condition known in the PDE literature as the oblique boundary condition, see Problem \eqref{eq:steinshape} for a precise statement. Intuitively, the idea is as follows. Suppose we want to compare two domains $\Omega_1$ and $\Omega_2$. Firstly, we expect the set of functions on $\Omega_1$ satisfying the Neumann boundary condition to characterize in some way the geometry of the domain $\Omega_1$, since by a classical probabilistic representation formula, these functions are related to the Brownian motion on $\Omega_1$ reflected orthogonally on the boundary. Second, we transport the unit normal vector from $\partial \Omega_2$ to the boundary of $\Omega_1$, and solve the oblique PDE problem where this new boundary vector replaces the normal vector of $\Omega_1$. We expect the set of all functions on $\Omega_1$ satisfying this new oblique condition to somehow represent Brownian motion on $\Omega_2$ reflected orthogonally on the boundary. The idea, then, is to use this set of functions as a means of quantifying the difference between the two domains.  In what follows, we focus on the case where $\Omega_1$ is the standard unit ball because we are interested in applications to isoperimetric problems, but we expect this strategy to work in broader cases, such as anisotropic perimeters.

Note that Stein's method is a new method for studying quantitative geometric inequalities. The four main methods for proving the stability of isoperimetric problems are the symmetrization method \cite{FMP08}, the transport mass approach \cite{figalli2010mass}, the selection principle \cite{cicalese2012selection} and the ABP method \cite{cabre2016sharp, cinti2022sharp}.
Note also that, while Stein's method requires more restrictive regularity assumptions, it makes all constants explicit, unlike the selection principle method, which proceeds via a contradiction argument.

Our first result is that in the case of a star-shaped domain $\Omega$ with a Hölder continuous boundary, this strategy works and allows us to show the following.
\begin{thm}\label{thm:mainthm}
Let $\Omega$ be a $\St^{2,\alpha}_{\kappa,\Lambda}$-domain and let $\nu$ denote its unit outward pointing normal vector. Then for some constant $C_{d,\alpha,\kappa,\Lambda}>0$ depending only on the dimension $d$ and the regularity parameters $\alpha, \kappa$ and $\Lambda$, the $\alpha$-Zolotarev distance between $\Omega$ and the ball is bounded as
\[
Z^\alpha(\Omega,B_1) \leq C_{d,\alpha,\kappa,\Lambda}\,\int_{\partial\Omega} \left| \frac{x}{|x|} - \nu \right|.
\]
\end{thm}
\noindent
The parameter $\kappa$ characterizes the fact that $\Omega$ is uniformly star-shaped, and the parameters $\alpha\in (0,1]$ and $\Lambda$ characterize the regularity of the boundary, see Definition $\ref{def:regulardomain}$ for details. The $\alpha$-Zolotarev distance is a dual distance between the uniform probability distributions on $B_1$ and on $\Omega$, which interpolates between the total variation distance for $\alpha=0$ and the $1$-Wasserstein distance for $\alpha=1$, see Equation $\eqref{def:holderdistanceshapes}$ for the precise definition. Note that the star-shaped assumption is needed in order for the oblique boundary PDE to actually make sense, see Section \ref{sec:geomsetting}. Note also the similarity between the quantity $\int_{\partial\Omega} | x/|x| - \nu|$ and the oscillation index given in \eqref{eq:oscillationindex}. Actually, this quantity represents a type of barycentric $L^1$-oscillation index, i.e. an $L^1$-oscillation index which only looks at the ball centered at the barycenter of $\Omega$. In other words, Theorem \ref{thm:mainthm} gives that the oscillation index \eqref{eq:oscillationindex} controls the $\alpha$-Zolotarev distances. Therefore, in combination with Fusco-Julin's classical results \cite{fuscojulin14}, Theorem \ref{thm:mainthm} gives a quantitative isoperimetric inequality with respect to the $\alpha$-Zolotarev distances. Our second result is a kernel reformulation of Stein's method for shapes. We say that a matrix-valued function $\tau_\Omega: \Omega \to \mathcal{M}_d(\R)$ is a Stein kernel for $\Omega$ (see Definition \ref{def:steinkernel}), when the following divergence-like formula
\[
\int_{\Omega} \langle \tau_\Omega, Du \rangle_{HS} = \int_{\partial\Omega} x\cdot u(x),
\]
holds for all $\C^1$ vector-valued functions $u:\bar{\Omega}\to \R^d$, where $\langle , \rangle_{HS}$ denotes the Hilbert-Schmidt inner product for matrices. It is easy to see from the divergence theorem that the constant matrix-valued function $\tau_{B_1}=I_d$, equals to the identity matrix, is a Stein kernel for the standard unit ball $B_1$. Our result is that in the case of a star-shaped domain $\Omega$ with a Hölder continuous boundary, the $L^1$ distance between $I_d$ and any Stein kernel of $\Omega$ controls the $\alpha$-Zolotarev distance to the ball.
\begin{thm}\label{thm:mainthmkernel}
Let $\Omega$ be a $\St^{2,\alpha}_{\kappa,\Lambda}$-domain. Assume that it admits a Stein kernel $\tau_\Omega : \bar{\Omega}\to \R^d$. If $\Lambda$ is small enough, then it exists a constant $C_{d,\alpha,\Lambda}>0$ depending only on the dimension $d$ and the regularity parameters $\alpha$ and $\Lambda$, such that
\[
Z^\alpha(\Omega,B_1) \leq C_{d,\alpha,\Lambda}\,\int_{\Omega} ||I_d-\tau_\Omega||_{HS}.
\]
\end{thm}
\noindent
Again, the parameters $\alpha$ and $\Lambda$ characterize the regularity of the boundary, see Definition $\ref{def:regulardomain}$ for details, and note that contrary to Theorem \ref{thm:mainthm}, the dependence of the constant in the parameter $\kappa$ on the uniform star-shaped condition has been removed. As an application, we get our third result, which is a quantitative form for the Brock-Weinstock inequality.
\begin{thm}\label{thm:mainstabBW}
If $\Lambda$ is small enough, then it exists a constant $C$ depending only on the dimension $d$ and the regularity parameters $\alpha$ and $\Lambda$, such that for any $\St^{2,\alpha}_{\kappa,\Lambda}$-domain $\Omega$ normalized such that $|\Omega|=|B_1|$,
\[
1 - \sigma_1(\Omega) \geq  C\,\frac{\sigma_1(\Omega)}{d|\Omega|} Z^\alpha(\Omega,B_1)^2
\]
where $\sigma_1$ is the first non-zero Steklov eigenvalue.
\end{thm}
\noindent
Note that the assumption that $\Lambda$ is small, which requires the domain $\Omega$ to be a small perturbation of the ball, is necessary to guarantee that the oblique PDE \eqref{eq:ksteinshape} is elliptic.
For a broader statement, see Theorem \ref{thm:stabBW}. Our two others applications are a quantitative isoperimetric inequality under a constraint on the first Steklov eigenvalue of the domain, and a quantitative isoperimetric inequality involving the sum of the two deficits in the weighted and unweighted isoperimetric inequalities.\newline\newline
{\textbf{Proposition.} (see Proposition \ref{prop:stabunderSteklovconstraint})\\
It exists a constant $C>0$ depending only on the dimension $d$ and the regularity parameters $\alpha, \kappa$ and $\Lambda$ of any $\St^{2,\alpha}_{\kappa,\Lambda}$-domain $\Omega$ centered as $\int_{\partial\Omega} x =0$, such that if the first non-zero Steklov eigenvalue satisfies $\sigma_1(\Omega) \geq 1 $, then 
\[
|\partial\Omega| \geq |\partial B_r| + C\,Z^\alpha(\Omega,B_1)^2,
\]
where $B_r$ is a ball of the same volume as $\Omega$.\newline\newline
}
{\textbf{Proposition.} (see Proposition \ref{prop:combinedstab})\\
It exists a constant $C>0$ depending only on the dimension $d$ and the regularity parameters $\alpha, \kappa$ and $\Lambda$ of any $\St^{2,\alpha}_{\kappa,\Lambda}$-domain $\Omega$ with the volume constraint $|\Omega|=|B_1|$, such that
\[
\delta(\Omega) + \delta_{|x|^2}(\Omega) \geq C\,Z^\alpha(\Omega,B_1)^2,
\]
where $\delta(\Omega) = |\partial \Omega | - |\partial B_1|$ denotes the usual isoperimetric deficit, and $\delta_{|x|^2}(\Omega) = \int_{\partial\Omega} |x|^2 - |\partial B_1| $ denotes the isoperimetric deficit with the weight function $|x|^2$.\newline\newline
}
It should be noted that these results appear to be new, even if they are no stronger than the quantitative isoperimetric inequalities already known from other proof methods.\newline\newline
We conclude this section by some notations. In all the sequel, we denote by $B_1$ the standard unit ball, i.e. the ball with center $0$ and radius $1$. To lighten the notation, we will omit to write the Hausdorff measure on integral on boundary, and therefore we make the convention that an integral on a boundary without any precision is always with respect to the $(d-1)$-dimensional Hausdorff measure. 

\section{Stein's method for shapes}\label{sec:Stein}

\subsection{Fundamentals of Stein's method}\label{sec:fundamStein}

Let us start by recalling the basic idea behind Stein's method, which comes from the field of mathematical statistics. For a more complete introduction, we refer the reader to the famous survey \cite{rosssurvey}, from which we borrow the title of this section. In \cite{stein1972bound}, Stein noted that if a random variable $W$ satisfies the following integration by parts formula for all sufficiently smooth functions $f:\R\rightarrow\R$,  $$\E\left[f'(W)-Wf(W)\right]=0,$$ then $W$ is distributed as the standard normal distribution $\gamma$, and moreover its distance to the normal is bounded in the following way
\begin{equation}\label{eq:Steinslemma}
W_1\left(\mathcal{L}(W),\gamma\right) \leq \underset{\underset{|f'|\leq \sqrt{2/\pi}}{|f|,|f''|\leq 2}}{\sup}\left|\E\left[f'(W)-Wf(W)\right]\right|
\end{equation}
where $W_1$ stands for the $L^1$-Wasserstein distance and $\mathcal{L}(W)$ denotes the distribution of $W$. What is now commonly referred to as Stein's method is the sum total of all the techniques developed to bound the supremum in \eqref{eq:Steinslemma}. The proof of \eqref{eq:Steinslemma} is based on the dual Kantorovich formulation for $W_1$,
\begin{equation}\label{eq:W1}
W_1\left(\mu,\gamma\right) = \underset{|h'|\leq 1}{\sup} \left|\int h\,d\mu - \int h\,d\gamma\right| 
\end{equation}
and on the fact that for all Lipschitz $h$, the ODE 
\begin{equation}\label{eq:Steineq}
f'(x)-xf(x) = h(x)-\int h\,d\gamma,\quad x\in\R 
\end{equation}
admits a solution $f$ satisfying $|f|,|f''|\leq 2|h'|_\infty$ and $|f'|\leq \sqrt{2/\pi}|h'|_\infty$. Hence, one can write
\[
W_1\left(\mu,\gamma\right) = \underset{|h'|\leq 1}{\sup} \left|\int h\,d\mu - \int h\,d\gamma\right|= \underset{f\,\mathrm{sol\, of} \eqref{eq:Steineq}}{\sup}\left| \int f'-xf\,d\mu\right| \leq \underset{\underset{|f'|\leq \sqrt{2/\pi}}{|f|,|f''|\leq 2}}{\sup} \left| \int f'-xf\,d\mu\right|
\]
proving Stein's inequality \eqref{eq:Steinslemma}. Note that this method can easily be extended to higher dimensions, and that it also admits an equivalent formulation by mean of the so-called Stein kernels, see e.g. \cite{nourdinpeccati}. A matrix-valued map $\tau:\R^d\to\R^{d\times d}$ is said to be a Stein kernel for the random variable $W\in\R^d$ if for all smooth enough vector-valued functions $f:\R^d\to\R^d$,
\[
\E\left[ \langle \tau(W), \Jac f \rangle_{HS} \right] = \E\left[ W\cdot f(W)\right],
\]
where $\langle,\rangle_{HS}$ denotes the Hilbert-Schmidt inner product of matrices. It is easy to see that if $W$ is distributed as a standard normal distribution, then the constant function equals to the identity matrix $\tau = I_d$ is a Stein kernel. Then, up to a multiplicative constant, the supremum in Inequality \eqref{eq:Steineq} can be reformulated as
\begin{equation}\label{eq:steindiscrepancyclassic}
\underset{\tau}{\inf}\, \E\, ||I_d-\tau(W)||_{HS}
\end{equation}
where the infimum runs over all Stein kernels for $W$. This quantity is known as the Stein discrependacy, and gives an upper bound for more classical distances between the distribution of $X$ and the Gaussian, such as $L^p$-Wasserstein distances. After adapting Stein's method for shapes in the next section, we will also give a kernel formulation in Section \ref{sec:steinkernel}. 

\subsection{Stein's method for shapes}\label{sec:geomsetting}

Our aim is to import Stein's method into a geometric context where probability distributions are replaced by shapes. For this purpose, let $\alpha\in(0,1]$, let $\Omega\subset\R^d$ be a $\C^{2,\alpha}$ domain, and let $\nu : \partial\Omega \to \mathbb{S}^{d-1}$ be the inner outward pointing normal of $\Omega$. We assume that $\Omega$ is star-shaped with respect to $0$, so one can find some function $R:\mathbb{S}^{d-1}\to (0,\infty)$ such that 
\[
\Omega = \{0\} \cup \left\{x\in\R^d\setminus \{0\}\,:\, 0<|x|<R\left(\frac{x}{|x|}\right) \right\}.
\]
Note that the inversion $x\mapsto x/|x|$ is a $\C^{2,\alpha}$ one-to-one correspondence between $\partial\Omega$ and $\mathbb{S}^{d-1}$ whose inverse is given by $\theta\in\mathbb{S}^{d-1}\mapsto  R(\theta)\theta\in\partial\Omega $. We can therefore define 
\begin{equation}\label{def:nuOmega}
\nu_\Omega : \mathbb{S}^{d-1}\to \mathbb{S}^{d-1},\quad \nu_\Omega(\theta) := \nu(R(\theta)\theta)
\end{equation}
which corresponds to the normal $\nu$ of $\partial\Omega$ transported on $\mathbb{S}^{d-1}$.
\begin{defi}\label{def:regulardomain}
We will say that a star-shaped domain $\Omega$ is in the regularity class $\St^{2,\alpha}_{\kappa,\Lambda}$, or is a $\St^{2,\alpha}_{\kappa,\Lambda}$-domain, if it satisfies the following
\begin{itemize}
\item it is a $\C^{2,\alpha}$ domain for $\alpha\in(0,1]$,
\item it is $\kappa$-uniformly star-shaped with respect to $0$ for some $\kappa\in (0,1]$, i.e.
\[
\forall x\in\partial\Omega,\quad \nu\cdot \frac{x}{|x|} \ge \kappa.
\]
\item the radius function $R$ is sufficiently close to $1$, i.e. for some $\Lambda>0$, it holds
\[
||R-1||_{\C^{1,\alpha}(\mathbb{S}^{d-1})} \leq \Lambda.
\]
\end{itemize}
\end{defi}
\noindent
Let us then define 
\begin{equation}\label{def:psidiffeo}
\psi : B_1 \to \Omega,\quad \psi(x)= R\left(\frac{x}{|x|}\right)x
\end{equation}
which is a $\C^1$-diffeomorphism whose inverse is given by $\psi^{-1}(x) = \frac{1}{R(x/|x|)}x$. 
We will denote by $J$ its Jacobian $J=|\det (\nabla \psi)|$. Note that $\psi$ restricted to $\mathbb{S}^{d-1}$ induces a diffeomorphism from $\mathbb{S}^{d-1}$ onto $\partial\Omega$ which is equal to the one mentioned above. Note also that it follows from the third point of Definition \ref{def:regulardomain} that $||J||_{\C^\alpha(B_1)}\leq \Lambda$.
Let then take a function $h\in C^\alpha(\R^d)$, and consider the following problem on the standard unit ball $B_1$
\begin{equation}\label{eq:steinshape}
\left\{
    \begin{array}{ll}
        \Delta f(x) = h(x)-\frac{1}{|\Omega|}\int_{\Omega}h, & x\in B_1 \\
        \nabla f(\theta) \cdot \nu_\Omega(\theta) = 0, & \theta\in \partial B_1
    \end{array}
\right.
\end{equation}
Equation \eqref{eq:steinshape} is a very particular instantiation of the more general theory of oblique PDEs, for which we refer the reader to the book \cite{lieberman2013oblique}. Note that the centering of $h$ is necessary since for $\Omega=B_1$ the problem boils down to the Neumann boundary problem on the ball.
The elliptic regularity theory ensures that \eqref{eq:steinshape} admits a solution $f\in \C^{2,\alpha}(B_1)$, and moreover, the following Schauder estimate holds for some constant $C_{d,\alpha,\kappa,\Lambda}>0$ depending only on the dimension $d$ and on the regularity parameters $\alpha$, $\kappa$ and $\Lambda$, (see Section \ref{sec:obliqueSchauder})
\begin{equation}\label{eq:schauderoblique}
||\nabla^2 f||_{\C^\alpha(B_1)} \leq C_{d,\alpha,\kappa,\Lambda}\, ||h||_{\C^\alpha(B_1)}
\end{equation}
Defining the set 
\begin{equation}\label{def:Halpha}
\mathcal{H}^\alpha=\left\{h\in \C^\alpha(\R^d),\,||h||_{\C^\alpha(\R^d)}\le 1\right\}
\end{equation}
we can then write that
\begin{align*}
\underset{h\in\mathcal{H}^\alpha}{\sup} \left|\int_{B_1} h\,dx - \frac{|B_1|}{|\Omega|}\int_{\Omega} h\,dx\right| &= \underset{f\,\mathrm{sol\, of} \eqref{eq:steinshape}}{\sup}\left| \int_{B_1}\Delta f\,dx\right|\\
& = \underset{f\,\mathrm{sol\, of} \eqref{eq:steinshape}}{\sup}\left| \int_{\mathbb{S}^{d-1}}\nabla f \cdot \theta\,d\theta\right|\\
& = \underset{f\,\mathrm{sol\, of} \eqref{eq:steinshape}}{\sup}\left| \int_{\mathbb{S}^{d-1}}J^{-1}\nabla f \cdot (\theta-R(\theta)\nu_\Omega(\theta))\,J\right|\\
&\leq  \underset{f\,\mathrm{sol\, of} \eqref{eq:steinshape}}{\sup} ||J^{-1}||_{\C^\alpha} ||\nabla^2 f||_{\C^\alpha(B_1)} \int_{\mathbb{S}^{d-1}} |\theta-R(\theta)\nu_\Omega(\theta)|\,J\,d\theta\\
&\leq C_{d,\alpha,\kappa,\Lambda}\,  \underset{h\in\mathcal{H}^\alpha}{\sup}||h||_{\C^\alpha(B_1)} \int_{\mathbb{S}^{d-1}} |\theta-R(\theta)\nu_\Omega(\theta)|\,J\,d\theta\\
&= \tilde{C}_{d,\alpha,\kappa,\Lambda}\, \int_{\mathbb{S}^{d-1}} |\theta-R(\theta)\nu_\Omega(\theta)|\,J\,d\theta.
\end{align*}
Therefore we have proved the following inequality, 
\begin{equation}\label{eq:ineqSteinshape}
Z^\alpha(\Omega,B_1) \leq C_{d,\alpha,\kappa,\Lambda}\,\int_{\mathbb{S}^{d-1}} |\theta-R(\theta)\nu_\Omega(\theta)|\,J\,d\theta
\end{equation}
for some constant $C_{d,\alpha,\kappa,\Lambda}$ depending only on $d,\alpha,\kappa$, and $\Lambda$, and where $Z^\alpha$ stands for the $\alpha$-Zolotarev distance, $\alpha\in (0,1]$, between $\Omega$ and the standard unit ball $B_1$, defined by
\begin{equation}\label{def:holderdistanceshapes}
Z^\alpha(\Omega,B_1) = \underset{h\in\mathcal{H}^\alpha}{\sup} \left|\int_{B_1} h\,dx - \frac{|B_1|}{|\Omega|}\int_{\Omega} h\,dx\right|
\end{equation}
with $\mathcal{H}^\alpha$ as defined in \eqref{def:Halpha}.
The $\alpha$-Zolotarev distance $Z^\alpha$ has the form of an integral probability metric, which is an usual type of dual distance between probability distributions. Note that for $\alpha=0$, $Z^\alpha$ coincides with the total variation distance, which coincides itself with the Fraenkel asymmetry \eqref{eq:fraenkel}, and for $\alpha=1$, it coincides with the Kantorovich distance \eqref{eq:W1}. Let us mention that the classic $k$-Zolotarev distance, $k\in \mathbb{N}^*$, is defined as an integral probability metric where the supremum is running over all functions with $\C^k$-norm bounded by $1$, see e.g. \cite{bobkov2023zolotarev}. We conclude this section by a few remarks.
\begin{rmq}
\,
\begin{itemize}
\item Note that Theorem \ref{thm:mainthm} is proved in the same way as Inequality \eqref{eq:ineqSteinshape} by only not writing $R(\theta)$ at line $3$ in the calculation above. We have chosen to not make $R$ appear in the statement of Theorem \ref{thm:mainthm} for the ease of interpretation, but we used the presence of $R$ in the applications of Sections \ref{ssec:stabsteklovconstraint} and \ref{ssec:stabcombinedisop}.

\item Inequality \eqref{eq:ineqSteinshape} is to be compared by analogy with Stein's inequality \eqref{eq:Steinslemma}. In particular the quantity 
\[
\underset{\underset{\mathrm{on }\, \mathbb{S}^{d-1}}{\nabla f\cdot \nu_\Omega = 0}}{\sup}\left| \int_{B_1}\Delta f\,dx\right|
\]
is a divergence-like quantity for $\Omega$ with respect to the ball $B_1$, in the spirit of the supremum in Stein's original lemma \eqref{eq:Steinslemma}. Note moreover that when $\Omega=B_1$, it follows from the divergence theorem and the fact that $\nu_{B_1}=\nu_\Omega=x$ is the unit normal vector of $\partial B_1$, that 
\[
\underset{\underset{\mathrm{on }\, \mathbb{S}^{d-1}}{\nabla f\cdot \nu_\Omega = 0}}{\sup}\left| \int_{B_1}\Delta f\,dx\right| = \underset{\underset{\mathrm{on }\, \mathbb{S}^{d-1}}{\nabla f\cdot \nu_\Omega = 0}}{\sup}\left| \int_{\mathbb{S}^{d-1}}\nabla f \cdot x\right| =0, 
\]
and so both terms of \eqref{eq:ineqSteinshape} are equal to zero if, and only if, $\Omega=B_1$.

\item Note that from Inequality \eqref{eq:ineqSteinshape} and the Cauchy-Schwarz inequality, it is easy to see that, for another constant $C_{d,\alpha,\kappa,\Lambda}$, we have 
\begin{equation}\label{eq:mainStein}
Z^\alpha(\Omega,B_1)^2 \leq C_{d,\alpha,\kappa,\Lambda}\,\int_{\mathbb{S}^{d-1}} |\theta-R(\theta)\nu_\Omega(\theta)|^2\,J\,d\theta
\end{equation}
\end{itemize}
\end{rmq}

\subsection{Oblique Schauder estimates}\label{sec:obliqueSchauder}

The goal of this section is to review some results about the oblique boundary PDE given in \eqref{eq:steinshape}. Let us recall it:
\[
\left\{
    \begin{array}{ll}
        \Delta f(x) = h(x)-\frac{1}{|\Omega|}\int_{\Omega}h, & x\in B_1 \\
        \nabla f(\theta) \cdot \nu_\Omega(\theta) = 0, & \theta\in \partial B_1
    \end{array}
\right.
\]
where $\nu_\Omega$, defined by \eqref{def:nuOmega}, is the normal vector of $\partial\Omega$ transported on $\mathbb{S}^{d-1}$, and $h$ is a $\alpha$-Hölder function defined on $\R^d$. The domain $\Omega\subset\R^d$ is assumed to be $\St^{2,\alpha}_{\kappa,\Lambda}$-regular, see Definition \ref{def:regulardomain}.
Such a problem is said to be oblique when the vector field of the boundary condition, here $\nu_\Omega$, points in the same direction as the outward normal vector of the domain, here $B_1$. In our case, this means that $\forall\theta\in\mathbb{S}^{d-1} $, $\theta\cdot \nu_\Omega(\theta)>0$, which is satisfied since by assumption $\Omega$ is uniformly star-shaped.
Note, however, that we are concerned here with the case where the second-order elliptic operator is the Laplacian, which is a very particular instantiation of the more general theory of oblique elliptic PDEs, for which we refer the reader to the book \cite{lieberman2013oblique}.
Although uniqueness cannot be guaranteed for such a problem since it is a generalization of a Neumann boundary problem, solvability is always ensured through the Fredholm alternative, see \cite[Theorem 6.31]{gilbargtrudinger} and the remark that follows.
As far as Schauder estimates are concerned, the assumption on the $\St^{2,\alpha}_{\kappa,\Lambda}$-regularity of $\Omega$, i.e. the Hölder regularity of the boundary, the uniform constant in the star-shaped condition, and the control of the Hölder norm of the $R$-radius function, are exactly the requirement of \cite [Thm 6. 30]{gilbargtrudinger} for the following oblique Schauder estimate to hold for a constant $C_{d,\alpha,\kappa,\Lambda}$ depending only on the dimension $d$, and the regularity parameters $\alpha$, $\kappa$ and $\Lambda$, and for any $h\in\C^{\alpha}(B_1)$,
\begin{equation}\label{eq:gilbargtrudingeroblique}
||\nabla^2 f||_{\C^{\alpha}(B_1)} \leq C_{d,\alpha,\kappa,\Lambda} \left(||f||_{\C^{0}(B_1)} + ||h||_{\C^{\alpha}(B_1)} \right)
\end{equation}
Note that, in general, Schauder's constant depends on the domain in a more complicated way, which is why we only consider the problem on the standard unit ball.
The proof proceeds as in the case of the Dirichlet boundary condition, by first computing the harmonic Green's function when the coefficients are constant on $\R^d_+$, and deducing interior regularity estimates, second, generalizing it to variable coefficients by the technique of freezing the coefficients, and third, generalizing it to curved boundaries by mean of local straightening of the boundary to eventually get the global Schauder estimate \eqref{eq:gilbargtrudingeroblique}. 
Let us mention that the regularity assumption on the boundary $\partial\Omega\in\C^{2,\alpha}$ can be relaxed to $\partial\Omega\in\C^{1,\alpha}$, see \cite{baderko1998schauder}.
We can now show how the oblique Schauder estimate \eqref{eq:schauderoblique} used in Section \ref{sec:geomsetting} can be derived from the estimate \eqref{eq:gilbargtrudingeroblique}. We are therefore going to prove that for some constant $C_{d,\alpha,\kappa,\Lambda}$ depending only on the dimension $d$, and the regularity parameters $\alpha$, $\kappa$ and $\Lambda$, and for any $h\in\C^{\alpha}(B_1)$,
\begin{equation}\label{eq:secschauderoblique}
||\nabla^2 f||_{\C^\alpha(B_1)} \leq C_{d,\alpha,\kappa,\Lambda}\, ||h||_{\C^\alpha(B_1)}
\end{equation}
The proof can be performed from \eqref{eq:gilbargtrudingeroblique} exactly as in the case of the Neumann boundary condition (see \cite{nardi2015schauder}), but let us sketch it for completeness. Arguing by contradiction, let $f_k\in\C^{2,\alpha}(\bar{B_1})$ and $h_k\in\C^\alpha(\bar{B_1})$ be two sequences normalized as
\[
\int_{B_1} f_k =0,\quad ||f_k||_{\C^{2,\alpha}(B_1)}=1,
\]
and satisfying for all $k\in\mathbb{N}$,
\[
\left\{
    \begin{array}{ll}
        \Delta f_k = h_k-\frac{1}{|B_1|}\int_{B_1}h_k, & x\in B_1 \\
        \nabla f_k \cdot \nu_\Omega = 0, & x\in \partial B_1
    \end{array}
\right.
\]
and
\[
||f_k||_{\C^{2,\alpha}(B_1)} > k \, ||h_k||_{\C^\alpha(B_1)}.
\]
It immediately follows that $h_k\to 0$ in $\C^\alpha(\bar{B_1})$. Moreover, using the Ascoli-Arzela theorem, we can obtain a subsequence $f_{k_n}\to \tilde{f}$ in $\C^2(\bar{B_1})$, and therefore deduce that $\tilde{f}\in\C^2(\bar{B_1})$ satisfies
\[
\left\{
    \begin{array}{ll}
        \Delta \tilde{f} = 0, & x\in B_1 \\
        \nabla \tilde{f} \cdot \nu_\Omega = 0, & x\in \partial B_1 \\
        \int_{B_1} \tilde{f} = 0 &
    \end{array}
\right.
\]
so necessarily $\tilde{f}=0$, from which we get a contradiction by using \eqref{eq:gilbargtrudingeroblique} and the normalization $||f_k||_{\C^{2,\alpha}(B_1)}=1$. Note that although we have reproduced the proof by contradiction for the sake of simplicity, the other two proofs in \cite[Section 4]{nardi2015schauder} can also be reproduced in the case of the oblique condition. In particular, even if they are complicated, all the constants can be made explicit.

\subsection{The probabilistic interpretation}

Whereas the Poisson equation with Neumann boundary conditions on $B_1$ can be represented as a Brownian particle moving in $B_1$ which is orthogonaly reflected each time it touches the boundary, the oblique Equation \eqref{eq:steinshape} can be represented as a Brownian particle which is reflected by the vector field $\nu_\Omega$ on the boundary. The star-shaped condition $x\cdot\nu>0$ then guarantees that the particle remains in $B_1$. Note that very close to a point $\theta\in\mathbb{S}^{d-1}$, the oblique particle behaves as if it were a Brownian with an orthogonal reflection inside the dilation $\Omega_{1/R(\theta)}=\{x\in\R^d\,;\, d(x,\Omega)\leq 1/R(\theta)\}$ whose boundary passes through $\theta$.
In this viewpoint, solutions of \eqref{eq:steinshape} can be written $$f(x)=- \int_0^\infty\E\left[h(X_t^x) \right]dt,$$ where $X_t^x$ solves the martingale problem
\[
X_t^x = x+ B_t -k_t, \quad k_t=\int_{0}^t \nu_{\Omega}(X_s)\,d|k|_s, \quad |k|_t =\int_0^t \textbf{1}_{X_t\in\mathbb{S}^{d-1}}\,d|k|_s
\]
with $B_t$ a standard Brownian motion on $\R^d$. When $\Omega$ is a smooth and strictly star-shaped domain, i.e. when $\forall x\in\partial\Omega,\, x\cdot\nu >0$, existence and uniqueness of $(X_t,k_t)$ are insured by Lions-Sznitman's result \cite{lions1984stochastic}. Note that the smoothness assumption can be relaxed by \cite{saisho1987stochastic}, and existence holds also in case of Lipschitz domains, see \cite{kwon1994reflected}. Let us conclude this section by mentioning that a probabilistic proof of Schauder's oblique estimate \eqref{eq:schauderoblique} is not known to the author, but seems to be an interesting avenue to explore in view of Stein's method for shapes.

\subsection{The Stein kernel formulation}\label{sec:steinkernel}

In this section, we state the kernel formulation of Stein's method for shapes. To achieve this, we need to make a few adjustments.
Let $\Omega$ be a $\St^{2,\alpha}_{\kappa,\Lambda}$-domain, and let $J$ be the Jacobian of the diffeomorphism $\psi:B_1\to \Omega$ defined in \eqref{def:psidiffeo} by $\psi(x) = R(x/|x|)x $.
Let us then take a function $h\in C^\alpha(\R^d)$, and consider the following variant of Problem \eqref{eq:steinshape} on the standard unit ball
\begin{equation}\label{eq:ksteinshape}
\left\{
    \begin{array}{ll}
        J(\theta)\,\langle I_d, \nabla^2 f(\theta)(D\psi)^{-1}(\theta) \rangle_{HS} = h(\theta)-\frac{1}{|\Omega|}\int_{\Omega}h, & \theta \in B_1 \\
        \nabla f(\theta) \cdot \psi (\theta) = 0, & \theta \in \partial B_1
    \end{array}
\right.
\end{equation}
where $(D\psi)^{-1}$ denotes the inverse of the Jacobian matrix of $\psi$, and $\langle\, , \, \rangle_{HS}$ denotes the Hilbert-Schmidt inner product for matrices.
If the radius function $R$ is close enough to $1$ in $\C^1$-norm, i.e. if $\Lambda$ is small enough in the notation of Definition \ref{def:regulardomain}, then the diffeomorphism $\psi$ is sufficiently close to the identity function in $\C^1$-norm to guarantee that the symmetric part of the matrix $D\psi$ is positive-definite, and therefore Problem \eqref{eq:ksteinshape} is elliptic. Moreover, since for all $\theta\in\mathbb{S}^{d-1}$, $\psi(\theta)\cdot \theta = R(\theta)|\theta|^2 = R(\theta)>0$, it follows that the boundary condition in Problem \eqref{eq:ksteinshape} is oblique. The regularity assumptions on $\Omega$ imply that the following Schauder estimate holds for some constant $C_{d,\alpha,\Lambda}>0$ depending only on the dimension $d$ and on the regularity parameters $\alpha$ and $\Lambda$,
\begin{equation}\label{eq:kernelobliqueschauder}
||\nabla^2 f||_{\C^\alpha(B_1)} \leq C_{d,\alpha,\Lambda}\, ||h||_{\C^\alpha(B_1)}
\end{equation}
Note that the dependance in $\kappa$ has been removed due to the change in boundary condition. Now the Schauder constant depends on the uniform lower bound of $\theta\cdot \psi(\theta)$, $\theta\in\mathbb{S}^{d-1}$, which is controlled by $\Lambda$, since $\theta\cdot \psi(\theta) = R(\theta)$.
The notion of Stein kernel can be translated into the geometric framework as follows.
\begin{defi}\label{def:steinkernel}
We will say that a matrix-valued function $\tau_\Omega: \Omega \to \mathcal{M}_d(\R)$ is a Stein kernel for $\Omega$, when for all $\C^1$ vector-valued functions $u:\bar{\Omega}\to \R^d$, it holds that
\[
\int_{\Omega} \langle \tau_\Omega, Du \rangle_{HS} = \int_{\partial\Omega} x\cdot u(x)
\] 
\end{defi}
\noindent
The divergence theorem immediately gives that the constant function $\tau_{B_1}=I_d$, equals to the identity matrix, is a Stein kernel for the standard unit ball. Stein's method can then be translated into the claim that the difference between any Stein kernel for $\Omega$ and the identity matrix $I_d$ should control some notion of distance between $\Omega$ and the ball. This is made rigorous by taking $f$ as a solution of \eqref{eq:steinshape} and computing
\begin{align*}
\int_{B_1} h - \frac{|B_1|}{|\Omega|}\int_{\Omega} h &= \int_{B_1} J(\theta)\,\langle I_d, \nabla^2 f(\theta) (D\psi)^{-1}(\theta) \rangle_{HS} \\
& =  \int_{B_1} J(\theta)\,\langle I_d, \nabla^2 f(\theta)D(\psi^{-1})(\psi(\theta)) \rangle_{HS} - \int_{\mathbb{S}^{d-1}} J(\theta)\nabla f(\theta)\cdot \psi(\theta) \\
& =  \int_{\Omega} \langle I_d, \nabla^2 f(\psi^{-1}(x)) D(\psi^{-1})(x) \rangle_{HS} - \int_{\partial\Omega} \nabla f(\psi^{-1}(x))\cdot x \\
& =  \int_{\Omega} \langle I_d, \nabla^2 f(\psi^{-1}(x)) D(\psi^{-1})(x) \rangle_{HS} - \int_{\Omega} \langle \tau_\Omega, D\left(\nabla f\circ \psi^{-1}(x) \right) \rangle_{HS} \\
& \leq \int_{\Omega} ||I_d-\tau_\Omega||_{HS} ||\nabla^2 f(\psi^{-1}(x)) D(\psi^{-1})(x) ||_{HS}\\
& \leq C_{d,\alpha,\Lambda}\, ||h||_{\C^\alpha(B_1)} \int_{\Omega} ||I_d-\tau_\Omega||_{HS}
\end{align*}
where we used the oblique boundary condition at line 2, the change of variable formula with the diffeomorphism $\psi$ at line 3, the definition of Stein kernel at line 4, the Cauchy-Schwarz inequality for the first inequality and the oblique Schauder estimate \eqref{eq:kernelobliqueschauder} for the second inequality. Recalling the definition \eqref{def:Halpha} of the $\alpha$-Zolotarev distance, we have therefore proved the following inequality
\begin{equation}\label{eq:ineqkernelSteinshape}
Z^\alpha(\Omega,B_1) \leq C_{d,\alpha,\Lambda}\,\int_{\Omega} ||I_d-\tau_\Omega||_{HS}
\end{equation}
which is Theorem \ref{thm:mainthmkernel}. We conclude this section by pointing out that the quantity 
\[
\underset{\tau_\Omega}{\inf} \int_{\Omega} ||I_d-\tau_\Omega||_{HS},
\]
where the infimum runs over all Stein kernels for $\Omega$, is the direct adaptation to the geometric framework of the so-called Stein discrepancy presented in Section \ref{sec:fundamStein}, see \cite{LNP} for its classical use with Gaussian distributions.

\section{Applications}

\subsection{Stability for the Brock-Weinstock inequality}

In this section, we will show how the Stein kernel formulation from Section \ref{sec:steinkernel} can be used to derive a stability result for the Brock-Weinstock inequality. We will consider the Sobolev trace inequality for $\C^1$-vector-valued functions $u : \bar{\Omega} \to \R^d$, on a domain $\Omega$, which states that
\begin{equation}\label{eq:Sobolevtrace}
\int_{\partial\Omega} ||u-\bar{u}||^2_2 \leq C\, \int_\Omega ||D u||^2_{HS}
\end{equation}
where $$\bar{u}:=\frac{1}{|\partial\Omega|}\int_{\partial\Omega} u,$$ and $||\cdot||_2$ denotes the Euclidean norm of vectors, and $||\cdot||_{HS}$ denotes the Hilbert-Schmidt norm for matrices. We will denote by $C_{BW}(\Omega)$ the best constant in \eqref{eq:Brock-Weinstock}, i.e. the smallest. Note that, as mentioned in the introduction, $1/C_{BW}(\Omega)$ coincides with the first non-zero Steklov eigenvalue $\sigma_1(\Omega)>0$, and furthermore 
\begin{equation}\label{eq:Brock-Weinstock}
C_{BW}(\Omega)\geq 1
\end{equation}
with equality attained if, and only if, $\Omega$ is a ball. We can now state our result, the proof of which follows the line of thought introduced in \cite{CFP} for the Poincaré constant of a probability distribution under moment constraints.

\begin{thm}\label{thm:stabBW}
Let $\Omega$ be a $\St^{2,\alpha}_{\kappa,\Lambda}$-domain, with $\Lambda$ small enough to guarantee that the positive part of $D\psi$ is positive-definite, and let $C_{BW}(\Omega)$ denote the inverse of its first non-zero Steklov eigenvalue. Assume that $\Omega$ is normalized in the following way
\begin{equation}\label{eq:normalizBW}
\int_{\partial\Omega}x=0,\quad \mathrm{and}\quad\int_{\partial\Omega}|x|^2 \geq d|\Omega|
\end{equation}
Then the following quantitative inequality holds for some constant $C>0$ depending only on the dimension $d$ and the regularity parameters $\alpha$ and $\Lambda$,
\begin{equation}\label{eq:stabBW}
C_{BW}(\Omega) \geq C_{BW}(B_1) + \frac{C}{d|\Omega|} Z^\alpha(\Omega,B_1)^2
\end{equation}
where $Z^\alpha$ denotes the $\alpha$-Zolotarev distance given in \eqref{def:holderdistanceshapes}.
\end{thm}
\noindent
It follows from the isoperimetric inequality with weight $|x|^2$, that if $\Omega$ is centered and satisfies $|\Omega|=|B_1|$, then the normalizing conditions \eqref{eq:normalizBW} are satisfied, see \cite{betta1999weighted}. Indeed,
\[
\int_{\partial\Omega}|x|^2 = \mathrm{Per}_{|x|^2}(\partial\Omega) \geq \int_{\mathbb{S}^{d-1}} |x|^2 = |\partial B_1| = d|B_1| = d|\Omega|.
\]
Consequently, we can deduce Theorem \ref{thm:mainstabBW} stated in Section \ref{sec:mainresults}.
\begin{corol}\label{corol:stabBW}
Let $\Omega$ be a $\St^{2,\alpha}_{\kappa,\Lambda}$-domain with $\Lambda$ small enough, and let $C_{BW}(\Omega)$ denote the inverse of its first non-zero Steklov eigenvalue. Assume that $\Omega$ satisfies $|\Omega|=|B_1|$. Then the following quantitative inequality holds for some constant $C>0$ depending only on the dimension $d$ and the regularity parameters $\alpha$ and $\Lambda$,
\[
C_{BW}(\Omega) \geq C_{BW}(B_1) + \frac{C}{d|\Omega|} Z^\alpha(\Omega,B_1)^2.
\]
\end{corol}
\noindent
\textit{Proof of Theorem \ref{thm:stabBW}.}
Let us consider the map
\[
\Phi : u \mapsto \int_{\partial\Omega} x \cdot u(x).
\]
The functional $\Phi$ is a linear form on the space 
\[
E := \left\{ u: \bar{\Omega}\to\R^d\,:\, \int_{\partial\Omega} u =0\,\,\mathrm{and}\, \int_{\Omega} \left| Du \right|_{HS}^2<\infty \right\},
\]
which is a Hilbert space equipped with the inner product $\langle f,g \rangle_E :=\int_{\Omega} \langle Df, Dg \rangle_{HS} $.
By using the Sobolev trace inequality \eqref{eq:Sobolevtrace} on $\Omega$, we obtain that $\Phi$ is continuous on $E$. Indeed,
\[
|\Phi(u)| = \left|\int_{\partial\Omega} x\cdot u(x)\right| 
 \leq \left(\int_{\partial\Omega} |x|^2 \int_{\partial\Omega}|u|^2 \right)^{1/2}
 \leq \left(\int_{\partial\Omega} |x|^2 \right)^{1/2} \left(C_{BW}(\Omega)\int_{\Omega} ||Du||_{HS}^2\right)^{1/2}
\]
Moreover, we get that $\Phi$ has an operator norm satisfying
\begin{equation}\label{eq:opnormbound}
|||\Phi|||_{op} \leq \left(C_{BW}(\Omega)\int_{\partial\Omega} |x|^2 \right)^{1/2}
\end{equation}
Hence, by using the Riesz representation theorem, one can find some $g\in E$ such that 
\[
\forall u\in E,\quad \Phi(u) = \int_{\Omega} \langle D g, D u \rangle_{HS}.
\]
In other words, $D g$ is a Stein kernel for $\Omega$ in the sense of Definition \ref{def:steinkernel}, and moreover, its norm is equal to the operator norm of $\Phi$, and therefore by \eqref{eq:opnormbound}, we get that
\begin{equation}\label{eq:ctrlnoyau}
\int_{B_1} \left|| D g| \right|_{HS}^2 \leq C_{BW}(\Omega)\int_{\partial\Omega} |x|^2
\end{equation}
We can therefore compute
\begin{align*}
\int_{\Omega} ||I_d-D g||_{HS}^2 &= d |\Omega| + \int_{\Omega} ||D g ||_{HS}^2 - 2\int_{\Omega} \langle I_d, D g \rangle_{HS} \\
& =  d |\Omega| + \int_{\Omega} ||D g ||_{HS}^2 - 2\int_{\partial\Omega} |x|_{HS}^2 \\
& \leq d |\Omega| + \left(C_{BW}(\Omega)-2\right)\int_{\partial\Omega} |x|_{HS}^2,
\end{align*}
where we used the fact that $g$ is a Stein kernel at line 2, and Inequality \eqref{eq:ctrlnoyau} at line 3. By using the normalization conditions \eqref{eq:normalizBW}, we obtain
\[
\int_{\Omega} ||I_d-D g||_{HS}^2 \leq \left( C_{BW}(\Omega) - 1 \right)d |\Omega|,
\]
which, in combination with \eqref{eq:ineqkernelSteinshape}, gives
\[
C_{BW}(\Omega) \geq C_{BW}(B_1) + \frac{Z^\alpha(\Omega,B_1)^2}{d|\Omega|C_{d,\alpha,\Lambda}^2}, 
\]
concluding the proof.
$\qed$

\subsection{Isoperimetric stability under a Steklov constraint}\label{ssec:stabsteklovconstraint}

In the previous section, we have seen how Stein's method for shapes can be used to prove stability for the first non-zero Steklov eigenvalue $\sigma_1=1/C_{BW}(\Omega)$ under a volume constraint. In this section, we study the converse and show that a quantitative isoperimetric inequality follows from a constraint on the first non-zero Steklov eigenvalue.

\begin{prop}\label{prop:stabunderSteklovconstraint}
Let $\Omega$ be a $\St^{2,\alpha}_{\kappa,\Lambda}$-domain centered such that $\int_{\partial\Omega} x =0$. If its first non-zero Steklov eigenvalue satisfies $\sigma_1 \geq 1 $, then the following quantitative isoperimetric inequality holds for some constant $C>0$ depending only on the dimension $d$ and the regularity parameters $\alpha, \kappa$ and $\Lambda$,
\begin{equation}\label{eq:stabunderSteklovconstraint}
|\partial\Omega| \geq |\partial B_r| + C\,Z^\alpha(\Omega,B_1)^2,
\end{equation}
where $Z^\alpha$ denotes the $\alpha$-Zolotarev distance given in \eqref{def:holderdistanceshapes}, and $B_r$ is a ball of the same volume as $\Omega$. 
\end{prop}
\begin{proof}
Recalling the definition of $J$ and $\nu_\Omega$ from Section \ref{sec:geomsetting}, one can see that
\begin{align*}
\int_{\mathbb{S}^{d-1}} |\theta-R(\theta)\nu_\Omega(\theta)|^2\,J\,d\theta & = \int_{\partial\Omega} \left| \frac{x}{|x|} - |x|\nu \right|^2 \\
& =  \int_{\partial\Omega} \left( 1 -2 x\cdot\nu + |x|^2 \right) \\
& = |\partial \Omega | - 2 \int_{\Omega} \nabla \cdot x + \int_{\partial\Omega} |x|^2\\
& = |\partial \Omega | - 2 d |\Omega| + \int_{\partial\Omega} |x|^2\\
& \leq |\partial \Omega | - 2 d |\Omega| + C_{BW}(\Omega)\int_\Omega ||I_d||_{HS}^2 \\
& = |\partial \Omega | + d |\Omega| \left(C_{BW}(\Omega) -2 \right)\\
& \leq |\partial \Omega | - d |\Omega|,
\end{align*}
where the Brock-Weinstock inequality \eqref{eq:Brock-Weinstock} is used at line $5$, and the constraint $\sigma=1/C_{BW}(\Omega)\geq 1$ is used at the last line. Inequality \eqref{eq:stabunderSteklovconstraint} follows then from Inequality \eqref{eq:mainStein}, and from the fact that $d|\Omega|=d|B_r|=|\partial B_r|$.
\end{proof}

\subsection{Stability for a combined weighted-unweighted isoperimetry}\label{ssec:stabcombinedisop}

In this section, we show how Stein's method for shapes immediately gives a quantitative isoperimetric inequality with a trade-off between the usual perimeter and the boundary momentum, i.e. the perimeter weighted by the squared norm $|x|^2$.
\begin{prop}\label{prop:combinedstab}
Let $\Omega$ be a $\St^{2,\alpha}_{\kappa,\Lambda}$-domain, with the volume constraint $|\Omega|=|B_1|$. Then the following quantitative combined weighted-unweighted isoperimetry holds for some constant $C>0$ depending only on the dimension $d$ and the regularity parameters $\alpha, \kappa$ and $\Lambda$,
\begin{equation}\label{eq:combinedstab}
\delta(\Omega) + \delta_{|x|^2}(\Omega) \geq C\,Z^\alpha(\Omega,B_1)^2,
\end{equation}
where $\delta(\Omega) = |\partial \Omega | - |\partial B_1|$ denotes the usual isoperimetric deficit, $\delta_{|x|^2}(\Omega) = \int_{\partial\Omega} |x|^2 - |\partial B_1| $ denotes the isoperimetric deficit with weight function $|x|^2$, and $Z^\alpha$ denotes the $\alpha$-Zolotarev distance given in \eqref{def:holderdistanceshapes}.
\end{prop}
\begin{proof}
By stoppping the calculation earlier in the proof of Proposition \ref{prop:stabunderSteklovconstraint}, we obtain
\begin{align*}
\int_{\mathbb{S}^{d-1}} |\theta-R(\theta)\nu_\Omega(\theta)|^2\,J\,d\theta & = \int_{\partial\Omega} \left| \frac{x}{|x|} - |x|\nu \right|^2 \\
& =  \int_{\partial\Omega} \left( 1 -2 x\cdot\nu + |x|^2 \right) \\
& = \left(|\partial \Omega | -  d |\Omega|\right) + \left( \int_{\partial\Omega} |x|^2 -d|\Omega| \right).
\end{align*}
Now, if $\Omega$ is normalized by the volume constraint $|\Omega|=|B_1|$, one recognizes the term $|\partial\Omega| - d|\Omega|$ as the usual isoperimetric deficit, and the term $\int_{\partial\Omega} |x|^2 - d|\Omega|$ as the weighted isoperimetric deficit. We then immediately get the result from Inequality \eqref{eq:mainStein}.
\end{proof}
\begin{rmq}\label{rmq:perturbation}
\,
\begin{itemize}
\item An immediate consequence of Proposition \ref{prop:combinedstab} is that it recovers a quantitative weighted or unweighted isoperimetric inequality as soon as the weighted and unweighted perimeters can be compared. Recalling the definition of $R$ in Section \ref{sec:geomsetting}, we can perform a perturbative comparison between those two perimeters. Write $R=1+\varepsilon$ for some small function $\varepsilon:\mathbb{S}^{d-1}\to \R_+$.
First, we want that $|\Omega|=|B_1|$ at order $2$, so we compute
\[
\left|\Omega\right| = \int_{\mathbb{S}^{d-1}} d\theta \int_0^{R(\theta)} t^{d-1}dt = \left|B_1\right| + \int_{\mathbb{S}^{d-1}} \varepsilon + \frac{d-1}{2}\int_{\mathbb{S}^{d-1}} \varepsilon^2 + o(\varepsilon^2)	
\]
and we impose $$\int_{\mathbb{S}^{d-1}} \varepsilon = - \frac{d-1}{2}\int_{\mathbb{S}^{d-1}} \varepsilon^2 $$  so that the volume is preserved at order $2$.
Second, we expand the weighted perimeter as
\begin{align*}
\int_{\partial\Omega} |x|^2 & = \int_{\mathbb{S}^{d-1}} R^{d+1} \sqrt{1 +\frac{1}{R^2}\left|\nabla_{\mathbb{S}^{d-1}}R\right|^2} \\
& = \int_{\mathbb{S}^{d-1}}\left( 1 + (d+1)\varepsilon + \frac{d(d+1)}{2}\varepsilon^2 + \frac{1}{2}\left|\nabla_{\mathbb{S}^{d-1}}\varepsilon\right|^2 \right) + o(\varepsilon^2,|\nabla\varepsilon|^2),
\end{align*}
and the usual perimeter as
\begin{align*}
\int_{\partial\Omega} 1 & = \int_{\mathbb{S}^{d-1}} R^{d-1} \sqrt{1 +\frac{1}{R^2}\left|\nabla_{\mathbb{S}^{d-1}}R\right|^2} \\
& = \int_{\mathbb{S}^{d-1}}\left( 1 + (d-1)\varepsilon + \frac{(d-1)(d-2)}{2}\varepsilon^2 + \frac{1}{2}\left|\nabla_{\mathbb{S}^{d-1}}\varepsilon\right|^2 \right) + o(\varepsilon^2,|\nabla\varepsilon|^2).
\end{align*}
So at order 2, for $\int_{\mathbb{S}^{d-1}} \varepsilon = - \frac{d-1}{2}\int_{\mathbb{S}^{d-1}} \varepsilon^2$, we have 
\[
|\partial\Omega| - \int_{\partial\Omega} |x|^2 = -d \int_{\mathbb{S}^{d-1}} \varepsilon^2  + o(\varepsilon^2,|\nabla\varepsilon|^2) 
\]
and therefore we can see that under a volume constraint, the weighted perimeter is larger than the usual perimeter at order $2$.
However, as can be seen by going further in the expansion, this inequality is not preserved at other orders.

\item If the domain $\Omega$ is convex and centered, with the same volume as the unit ball $|\Omega|=|B_1|$, then its boundary momentum is always larger than its usual unweighted perimeter:
\[
\int_{\partial\Omega} |x|^2 \geq |\partial\Omega|.
\] 
This inequality has been proved by Bucur, Ferone, Nitsch and Trombetti in \cite{bucur2021weinstock} and constitutes a refinement of Brock's inequality \cite{brock2001isoperimetric}. It should be noted that, as indicated in \cite{bucur2021weinstock} and in accordance with the perturbative expansion of the previous item, the convexity assumption is necessary to obtain this inequality. As an immediate consequence of Proposition \ref{prop:combinedstab}, we get a quantitative version of the weighted isoperimetric inequality \eqref{eq:weightedisop} restricted to convex domains.

\item On this subject, let us point out that in the case of a constraint on the perimeter instead of the volume, the boundary momentum is maximized by the balls in dimension $2$, but this is no longer the case in higher dimension, see \cite[Section 3]{gavitone2020quantitative} and \cite[Section 4]{la2023some}, where another functional is introduced to renormalize the boundary momentum.
\end{itemize}
\end{rmq}
\noindent
\textbf{Aknowledgements}\\
I would like to thank Xavier Lamy for the many discussions that helped me to clarify my ideas each time. I would also like to thank Domenico Angelo La Manna for pointing out the references \cite{fusco2023some,gavitone2020quantitative,la2023some}, and also for pointing out an inaccuracy in the expansion of Remark \ref{rmq:perturbation}.

\bibliographystyle{plain}
\bibliography{biblio}

\begin{thebibliography}{10}

\bibitem{baderko1998schauder}
Elena~A Baderko.
\newblock Schauder estimates for oblique derivative problems.
\newblock {\em Comptes Rendus de l'Acad{\'e}mie des Sciences-Series
  I-Mathematics}, 326(12):1377--1380, 1998.

\bibitem{betta1999weighted}
Maria~Francesca Betta, Friedemann Brock, Anna Mercaldo, and Maria~Rosaria
  Posteraro.
\newblock A weighted isoperimetric inequality and applications to
  symmetrization.
\newblock {\em Journal of Inequalities and Applications}, 1999(3):670245, 1999.

\bibitem{bhattacharya2001some}
Tilak Bhattacharya.
\newblock {Some observations on the first eigenvalue of the p-Laplacian and its
  connections with asymmetry}.
\newblock 2001.

\bibitem{weitsmanbhattacharya}
Tilak Bhattacharya and Allen Weitsman.
\newblock {Estimates for Green’s function in terms of asymmetry}.
\newblock {\em AMS Contemporary Math. Series}, 221:31--58, 1999.

\bibitem{bobkov2023zolotarev}
Sergey~G Bobkov.
\newblock Zolotarev-type distances.
\newblock {\em Preprint}, 2023.

\bibitem{bobkov2009weighted}
Sergey~G Bobkov and Michel Ledoux.
\newblock {On weighted isoperimetric and Poincaré-type inequalities}.
\newblock In {\em High dimensional probability V: the Luminy volume}, volume~5,
  pages 1--30. Institute of Mathematical Statistics, 2009.

\bibitem{borell1975brunn}
Christer Borell.
\newblock {The Brunn-Minkowski inequality in Gauss space}.
\newblock {\em Inventiones mathematicae}, 30(2):207--216, 1975.

\bibitem{brasco2016spectral}
Lorenzo Brasco and Guido De~Philippis.
\newblock Spectral inequalities in quantitative form.
\newblock {\em Shape optimization and spectral theory}, pages 201--281, 2017.

\bibitem{BrascoDePhilippisRuffini}
Lorenzo Brasco, Guido De~Philippis, and Berardo Ruffini.
\newblock {Spectral optimization for the Stekloff--Laplacian: the stability
  issue}.
\newblock {\em Journal of Functional Analysis}, 262(11):4675--4710, 2012.

\bibitem{10.1215/00127094-3120167}
Lorenzo Brasco, Guido~De Philippis, and Bozhidar Velichkov.
\newblock {Faber–Krahn inequalities in sharp quantitative form}.
\newblock {\em Duke Mathematical Journal}, 164(9):1777 -- 1831, 2015.

\bibitem{brock2001isoperimetric}
Friedemann Brock.
\newblock {An isoperimetric inequality for eigenvalues of the Stekloff
  problem}.
\newblock {\em ZAMM-Journal of Applied Mathematics and Mechanics/Zeitschrift
  f{\"u}r Angewandte Mathematik und Mechanik: Applied Mathematics and
  Mechanics}, 81(1):69--71, 2001.

\bibitem{bucur2021weinstock}
Dorin Bucur, Vincenzo Ferone, Carlo Nitsch, and Cristina Trombetti.
\newblock Weinstock inequality in higher dimensions.
\newblock {\em Journal of Differential Geometry}, 118(1):1--21, 2021.

\bibitem{cabre2016sharp}
Xavier Cabré, Xavier Ros-Oton, and Joaquim Serra.
\newblock {Sharp isoperimetric inequalities via the ABP method}.
\newblock {\em Journal of the European Mathematical Society (EMS Publishing)},
  18(12), 2016.

\bibitem{canete2010some}
Antonio Canete, Michele Miranda, and Davide Vittone.
\newblock Some isoperimetric problems in planes with density.
\newblock {\em Journal of Geometric Analysis}, 20:243--290, 2010.

\bibitem{chambers2019proof}
Gregory~R Chambers.
\newblock Proof of the log-convex density conjecture.
\newblock {\em Journal of the European Mathematical Society (EMS Publishing)},
  21(8), 2019.

\bibitem{cianchi2009sharp}
Andrea Cianchi, Nicola Fusco, Francesco Maggi, and Aldo Pratelli.
\newblock {The sharp Sobolev inequality in quantitative form}.
\newblock {\em Journal of the European Mathematical Society}, 11(5):1105--1139,
  2009.

\bibitem{cicalese2012selection}
Marco Cicalese and Gian~Paolo Leonardi.
\newblock A selection principle for the sharp quantitative isoperimetric
  inequality.
\newblock {\em Archive for Rational Mechanics and Analysis}, 206(2):617--643,
  2012.

\bibitem{cinti2022sharp}
Eleonora Cinti, Federico Glaudo, Aldo Pratelli, Xavier Ros-Oton, and Joaquim
  Serra.
\newblock Sharp quantitative stability for isoperimetric inequalities with
  homogeneous weights.
\newblock {\em Transactions of the American Mathematical Society},
  375(3):1509--1550, 2022.

\bibitem{colbois2024some}
Bruno Colbois, Alexandre Girouard, Carolyn Gordon, and David Sher.
\newblock {Some recent developments on the Steklov eigenvalue problem}.
\newblock {\em Revista Matem{\'a}tica Complutense}, 37(1):1--161, 2024.

\bibitem{courtadefathi20}
Thomas~A Courtade and Max Fathi.
\newblock {Stability of the Bakry-Émery theorem on $\mathbb{R}^n$}.
\newblock {\em Journal of Functional Analysis}, 279(2):108523, 2020.

\bibitem{CFimprovedLichn}
Thomas~A Courtade and Max Fathi.
\newblock {Stability of Klartag's improved Lichnerowicz inequality}.
\newblock {\em arXiv preprint arXiv:2404.12277}, 2024.

\bibitem{CFPoincareKorn}
Thomas~A Courtade and Max Fathi.
\newblock {Stability of the Poincaré-Korn inequality}.
\newblock {\em arXiv preprint arXiv:2405.01441}, 2024.

\bibitem{CFP}
Thomas~A. Courtade, Max Fathi, and Ashwin Pananjady.
\newblock {Existence of Stein kernels under a spectral gap, and discrepancy
  bounds}.
\newblock {\em Annales de l'Institut Henri Poincaré, Probabilités et
  Statistiques}, 55(2):777 -- 790, 2019.

\bibitem{csato}
Gyula Csató.
\newblock On the isoperimetric problem with perimeter density $r^p$.
\newblock {\em Communications on Pure and Applied Analysis}, 17(6):2729--2749,
  2018.

\bibitem{faber1923beweis}
Georg Faber.
\newblock {Beweis, dass unter allen homogenen Membranen von gleicher Fl{\"a}che
  und gleicher Spannung die kreisf{\"o}rmige den tiefsten Grundton gibt}.
\newblock 1923.

\bibitem{FGS}
Max Fathi, Ivan Gentil, and Jordan Serres.
\newblock Stability estimates for the sharp spectral gap bound under a
  curvature-dimension condition.
\newblock {\em Annales de l'Institut Fourier}, 74(6):2425--2459, 2024.

\bibitem{figalli2013stability}
Alessio Figalli.
\newblock Stability in geometric and functional inequalities.
\newblock In {\em European congress of mathematics}, pages 585--599. Citeseer,
  2013.

\bibitem{figalli2010mass}
Alessio Figalli, Francesco Maggi, and Aldo Pratelli.
\newblock A mass transportation approach to quantitative isoperimetric
  inequalities.
\newblock {\em Inventiones mathematicae}, 182(1):167--211, 2010.

\bibitem{fuglede1989stability}
Bent Fuglede.
\newblock Stability in the isoperimetric problem for convex or nearly spherical
  domains in $\mathbb{R}^n$.
\newblock {\em Transactions of the American Mathematical Society},
  314(2):619--638, 1989.

\bibitem{fusco2015quantitativesurvey}
Nicola Fusco.
\newblock The quantitative isoperimetric inequality and related topics.
\newblock {\em Bulletin of Mathematical Sciences}, 5:517--607, 2015.

\bibitem{fuscojulin14}
Nicola Fusco and Vesa Julin.
\newblock A strong form of the quantitative isoperimetric inequality.
\newblock {\em Calculus of Variations and Partial Differential Equations},
  50:925--937, 2014.

\bibitem{fusco2023some}
Nicola Fusco and Domenico~Angelo La~Manna.
\newblock Some weighted isoperimetric inequalities in quantitative form.
\newblock {\em Journal of Functional Analysis}, 285(2):109946, 2023.

\bibitem{FMP08}
Nicola Fusco, Francesco Maggi, and Aldo Pratelli.
\newblock The sharp quantitative isoperimetric inequality.
\newblock {\em Annals of mathematics}, pages 941--980, 2008.

\bibitem{fusco2009stabilityFK}
Nicola Fusco, Francesco Maggi, and Aldo Pratelli.
\newblock {Stability estimates for certain Faber-Krahn, isocapacitary and
  Cheeger inequalities}.
\newblock {\em Annali della Scuola Normale Superiore di Pisa-Classe di
  Scienze}, 8(1):51--71, 2009.

\bibitem{gavitone2020quantitative}
Nunzia Gavitone, Domenico~Angelo La~Manna, Gloria Paoli, and Leonardo Trani.
\newblock {A quantitative Weinstock inequality for convex sets}.
\newblock {\em Calculus of Variations and Partial Differential Equations},
  59(1):2, 2020.

\bibitem{gilbargtrudinger}
David Gilbarg and Neil Trudinger.
\newblock {\em Elliptic partial differential equations of second order}, volume
  224.
\newblock Springer, 1977.

\bibitem{hall1992quantitative}
R.R. Hall.
\newblock A quantitative isoperimetric inequality in n-dimensional space.
\newblock {\em Journal für die reine und angewandte Mathematik}, 428:161--176,
  1992.

\bibitem{hansen1994isoperimetric}
Wolfhard Hansen and Nikolai Nadirashvili.
\newblock Isoperimetric inequalities in potential theory.
\newblock {\em Potential Analysis}, 3:1--14, 1994.

\bibitem{krahn1925rayleigh}
Edgar Krahn.
\newblock {{\"U}ber eine von Rayleigh formulierte Minimaleigenschaft des
  Kreises}.
\newblock {\em Mathematische Annalen}, 94(1):97--100, 1925.

\bibitem{kuznetsov2014legacy}
Nikolay Kuznetsov, Tadeusz Kulczycki, M~Kwa{\'s}nicki, Alexander Nazarov,
  Sergey Poborchi, Iosif Polterovich, and Bart{\l}omiej Siudeja.
\newblock {The legacy of Vladimir Andreevich Steklov}.
\newblock {\em Notices of the AMS}, 61(1):190, 2014.

\bibitem{kwon1994reflected}
Youngmee Kwon.
\newblock {Reflected Brownian motion in Lipschitz domains with oblique
  reflection}.
\newblock {\em Stochastic processes and their applications}, 51(2):191--205,
  1994.

\bibitem{la2023some}
Domenico~Angelo La~Manna and Rossano Sannipoli.
\newblock Some isoperimetric inequalities involving the boundary momentum.
\newblock {\em arXiv preprint arXiv:2309.14191}, 2023.

\bibitem{LNP}
Michel Ledoux, Ivan Nourdin, and Giovanni Peccati.
\newblock Stein’s method, logarithmic {S}obolev and transport inequalities.
\newblock {\em Geometric and Functional Analysis}, 25:256--306, 2015.

\bibitem{lieberman2013oblique}
Gary~M Lieberman.
\newblock {\em Oblique derivative problems for elliptic equations}.
\newblock World Scientific, 2013.

\bibitem{lions1984stochastic}
Pierre-Louis Lions and Alain-Sol Sznitman.
\newblock Stochastic differential equations with reflecting boundary
  conditions.
\newblock {\em Communications on pure and applied Mathematics}, 37(4):511--537,
  1984.

\bibitem{melas1992stability}
Antonios~D Melas.
\newblock The stability of some eigenvalue estimates.
\newblock {\em Journal of Differential Geometry}, 36(1):19--33, 1992.

\bibitem{milman2015sharp}
Emanuel Milman.
\newblock Sharp isoperimetric inequalities and model spaces for the
  curvature-dimension-diameter condition.
\newblock {\em Journal of the European Mathematical Society}, 17(5):1041--1078,
  2015.

\bibitem{morgan2005manifolds}
Frank Morgan.
\newblock Manifolds with density.
\newblock {\em Notices of the AMS}, 52(8):853--858, 2005.

\bibitem{morgan2013existence}
Frank Morgan and Aldo Pratelli.
\newblock Existence of isoperimetric regions in $\mathbb{R}^n$ with density.
\newblock {\em Annals of Global Analysis and Geometry}, 43(4):331--365, 2013.

\bibitem{nadirashvili1996conformal}
Nikolai Nadirashvili.
\newblock {Conformal maps and isoperimetric inequalities for eigenvalues of the
  Neumann problem}.
\newblock In {\em Proceedings of the Ashkelon Workshop on Complex Function
  Theory}, volume~11, pages 197--201, 1996.

\bibitem{nardi2015schauder}
Giacomo Nardi.
\newblock {Schauder estimate for solutions of Poisson’s equation with Neumann
  boundary condition}.
\newblock {\em L’enseignement Math{\'e}matique}, 60(3):421--435, 2015.

\bibitem{nobili-violo}
Francesco Nobili and Ivan~Yuri Violo.
\newblock {Stability of Sobolev inequalities on Riemannian manifolds with Ricci
  curvature lower bounds}.
\newblock {\em Advances in Mathematics}, 440:109521, 2024.

\bibitem{nourdinpeccati}
Ivan Nourdin and Giovanni Peccati.
\newblock {\em {Normal approximations with Malliavin calculus: from Stein's
  method to universality}}, volume 192.
\newblock Cambridge University Press, 2012.

\bibitem{povel1999confinement}
Tobias Povel.
\newblock {Confinement of Brownian motion among Poissonian obstacles in
  $\mathbb{R}^d$, $d\geq 3$}.
\newblock {\em Probability theory and related fields}, 114:177--205, 1999.

\bibitem{rosales2008isoperimetric}
C{\'e}sar Rosales, Antonio Canete, Vincent Bayle, and Frank Morgan.
\newblock {On the isoperimetric problem in Euclidean space with density}.
\newblock {\em Calculus of Variations and Partial Differential Equations},
  31(1):27--46, 2008.

\bibitem{rosssurvey}
Nathan Ross.
\newblock {Fundamentals of Stein’s method}.
\newblock {\em Probability Surveys}, 8:210 -- 293, 2011.

\bibitem{saisho1987stochastic}
Yasumasa Saisho.
\newblock Stochastic differential equations for multi-dimensional domain with
  reflecting boundary.
\newblock {\em Probability Theory and Related Fields}, 74(3):455--477, 1987.

\bibitem{serresSPA}
Jordan Serres.
\newblock Stability of higher order eigenvalues in dimension one.
\newblock {\em Stochastic Processes and their Applications}, 155:459--484,
  2023.

\bibitem{jordan2023stability}
Jordan Serres.
\newblock {Stability of the Poincaré constant}.
\newblock {\em Bernoulli}, 29(2):1297--1320, 2023.

\bibitem{stein1972bound}
Charles Stein.
\newblock A bound for the error in the normal approximation to the distribution
  of a sum of dependent random variables.
\newblock In {\em Proceedings of the Sixth Berkeley Symposium on Mathematical
  Statistics and Probability}, volume~6, pages 583--603, 1972.

\bibitem{sznitman1997fluctuations}
Alain-Sol Sznitman.
\newblock Fluctuations of principal eigenvalues and random scales.
\newblock {\em Communications in mathematical physics}, 189(2):337--363, 1997.

\bibitem{weinstock1954inequalities}
Robert Weinstock.
\newblock Inequalities for a classical eigenvalue problem.
\newblock {\em Journal of Rational Mechanics and Analysis}, 3:745--753, 1954.

\end{thebibliography}
\end{document}